\documentclass{amsart}
\usepackage{amssymb,amsmath,latexsym}
\usepackage{amsthm}
\usepackage{fontenc}
\usepackage{amssymb}
\numberwithin{equation}{section}

\newtheorem{theorem}{Theorem}[section]
\newtheorem{corollary}{Corollary}[theorem]

\begin{document}
\author{Alexander E Patkowski}
\title{A note on arithmetic diophantine series}

\maketitle
\begin{abstract}We consider some asymptotic analysis for series related to the work of Hardy and Littlewood on Diophantine approximation, as well as Davenport. In particular, we expand on ideas from some previous work on arithmetic series and the RH. \end{abstract}

% AMS keywords (used in AMS journals)

% AMS subject classifications (used in AMS journals)

\section{Introduction} 
In a 1923 paper by Hardy and Littlewood [4], we find some mention of the series 
$$\sum_{n\ge1}\frac{\bar{B}_m(nx)}{n^{s}},$$
with $\sigma:=\Re(s)>1,$ in the setting of analysis on problems of Diophantine approximation. Here $\bar{B}_m(x):=\sum_{j\ge0}^{m}\binom{m}{j}B_{m-j}\{x\}^j,$ where $B_j$ is the $j$th Bernoulli number, and $\{x\}=x-[x],$ $[x]$ being the floor function. Not long after, Davenport's famous work [2] was published showing interesting properties on arithmetic series of the form
\begin{equation}\sum_{n\ge1}\frac{a_n\bar{B}_1(nx)}{n},\end{equation}
where $a_n$ is taken to be a multiplicative arithmetic function $a:\mathbb{N}\rightarrow\mathbb{C}.$ Here and throughout the paper we will take the set of natural numbers $\mathbb{N}$ to exclude $0,$ and write $\mathbb{N}_0$ to mean non-negative integers. Recall that the M$\ddot{o}$bius function is denoted by $\mu(n)$ [10]. The series (1.1) has been explored to a great extent [1, 6, 9]. \par One of our main results is given in the following.
\begin{theorem} Let $k\ge1$ be a natural number, and let $\Upsilon_k(x)$ be a polynomial of degree $k$ plus a term of the form $\dot{h}\log(x)x^{k},$ where $\dot{h}$ is a computable constant. Put $C_k=\frac{k!}{(2\pi i)^k}(1+(-1)^k)\frac{\zeta'(k)}{\zeta(k)},$ when $k>1,$ and $C_1=0.$ We have that the Riemann Hypothesis is equivalent to 
\begin{equation}\sum_{n\ge1}\frac{\mu(n)\log(n)}{n^k}\bar{B}_{k}(nx)=C_k+\Upsilon_{k-1}(x)+O(x^{k-\frac{1}{2}}),\end{equation}
as $x\rightarrow0^{+}.$
\end{theorem}
\begin{proof} First, we note from [5, eq.(4.16)] that for $m\ge1,$
\begin{equation}\bar{B}_m(x)=-m!\sum_{n\neq0}(2\pi i n)^{-m}e^{2\pi i n x}.\end{equation}
Now, it is well-known [1, eq.(5.12)] that for $0<c<1,$
\begin{equation}\frac{1}{2\pi i}\int_{(c)}e^{s(i\pi/2-\log(2\pi)-\log(nx))}\Gamma(s)ds=e^{2\pi i nx}.\end{equation}
This integral is also noted in [8, pg.91, pg.406]. We will follow similar lines as [1] in constructing our integral. Combining (1.3) with (1.4), we may sum the desired series, by conditional convergence, to get
\begin{equation}-\frac{m!}{2\pi i}\int_{(c)}(e^{s(i\pi/2)}+(-1)^me^{-s(i\pi/2)})e^{-s(\log(2\pi)+\log(x))}\zeta(s+m)\Gamma(s)ds=(2\pi i)^m\bar{B}_m(x).\end{equation}
Here the gamma factor $\Gamma(s)$ is estimated by Stirling's formula [5, pg.151, eq.(5.112)] when $s=\sigma+it,$ $t\neq0,$ and $\zeta(s+m)$ is also bounded on the line $s=\sigma+it$ when $m\ge1,$ since [10, pg.95]
$$\zeta(s)=O(|t|^{k}),$$ for any fixed $\sigma_0>0$ and $\Re(s)>\sigma_0.$ Note that the $m$ even case corresponds to the Mellin transform of $\cos(t)$ and the $m$ odd case of this integral corresponds to the Mellin transform of $\sin(t),$ both of which are valid when $0<\Re(s)<1.$ \par
Now using the formula $-\sum_{n\ge1}\mu(n)\log(n)/n^s=\zeta'(s)/\zeta^2(s),$ $\Re(s)\ge1,$ we may again invert to get
\begin{equation}\frac{m!}{2\pi i}\int_{(c)}(e^{s(i\pi/2)}+(-1)^me^{-s(i\pi/2)})e^{-s(\log(2\pi)+\log(x))}\frac{\zeta'(s+m)}{\zeta(s+m)}\Gamma(s)ds\end{equation}
$$=(2\pi i)^m\sum_{n\ge1}\frac{\mu(n)\log(n)\bar{B}_m(nx)}{n^m}.$$
Define $K(s):=(e^{s(i\pi/2)}+(-1)^me^{-s(i\pi/2)}),$ and note that $|K(s)|\ll e^{|t|\pi/2}.$ In particular we see that 
$$\Bigg|\int_{c-i(T+\frac{1}{\log^2(T)})}^{c-iT}K(s)e^{-s(\log(2\pi)+\log(x))}\frac{\zeta'(s+m)}{\zeta(s+m)}\Gamma(s)ds\Bigg|$$
$$\begin{aligned}= e^{m(\log(2\pi)+\log(x))}\\
&\times \Bigg|\int_{c+m-i(T+\frac{1}{\log^2(T)})}^{c+m-iT}K(s-m)e^{-s(\log(2\pi)+\log(x))}\frac{\zeta'(s)}{\zeta(s)}\Gamma(s-m)ds\Bigg| \\ \end{aligned}$$
$$\gg e^{m(\log(2\pi)+\log(x))}\int_{-(T+\frac{1}{\log^2(T)})}^{-T}\bigg|K(s-m)e^{-s(\log(2\pi)+\log(x))}\frac{\zeta'(s)}{\zeta(s)}\Gamma(s-m)\bigg|dt$$
$$\gg e^{(m-\sigma)(\log(2\pi)+\log(x))}\left(\frac{\log(2)}{\cosh(\sigma\log(2))}+\frac{\zeta'(\sigma)}{\zeta(\sigma)}\right)T^{\sigma-\frac{1}{2}-m}\int_{-(T+\frac{1}{\log^2(T)})}^{-T}dt $$
$$=e^{(m-\sigma)(\log(2\pi)+\log(x))}\left(\frac{\log(2)}{\cosh(\sigma\log(2))}+\frac{\zeta'(\sigma)}{\zeta(\sigma)}\right)\frac{T^{\sigma-\frac{1}{2}-m}}{\log^2(T)}.$$
Here we have made the change of variable $s\rightarrow s-m$ and estimated $|\frac{\zeta'(s)}{\zeta(s)}|$ using arguments from [10, Theorem 11.5(A)]. Hence, we require $\sigma-\frac{1}{2}-m<0$ for convergence. Since we know $0<\sigma<1,$ our condition $m\ge1$ is sufficient. (For similar examples and arguments related to algebraic decay see [8, pg.127].) We integrate over the positively oriented rectangle with corners $(c,iT),$ $(-M-\frac{1}{2},iT),$ $(-M-\frac{1}{2},-iT),$ and $(c,-iT),$ and sufficiently large $M>0.$ We move the line of integration of (1.6) to the left and compute the residues of the poles at the non-trivial zeros $s=-m+\rho,$ the residue at the pole $s=0,$ and poles at the negative integers. We compute the residues when $s=-l,$ for $l< m,$ giving the polynomial of degree $m-1$ plus the term $\dot{h}\log(x)x^{m-1}$ arising from the double pole at $s=-m+1$ (the $\Upsilon_{m-1}(x)).$ This term is included since $\log(x)x^{m-1}\le x^{m-1/2}$ implies $\log(x)\le x^{1/2},$ which is valid when $x\in(0,\infty).$ We find,
$$(2\pi i)^m\sum_{n\ge1}\frac{\mu(n)\log(n)\bar{B}_m(nx)}{n^m}=C_m+ \Upsilon_{m-1}(x)$$
$$+m!\sum_{\rho}K(\rho-m)e^{-(\rho-m)(\log(2\pi)+\log(x))}\Gamma(\rho-m)$$
$$+\frac{m!}{2\pi i}\int_{(d)}K(s)e^{-s(\log(2\pi)+\log(x))}\frac{\zeta'(s+m)}{\zeta(s+m)}\Gamma(s)ds,$$
with $C_m=m!(1+(-1)^m)\frac{\zeta'(m)}{\zeta(m)},$ when $m>1,$ $C_1=0,$
where $-m<d<-m+\frac{1}{2}.$ Next we consider when $l\ge m.$ Computing the residues at the double poles $s=-m-2l,$ $l\in\mathbb{N}_0,$ give rise to a series over $l$ of the form $\sum_{l\ge0}(p_l+r_l\log(x))x^{m+2l}.$ The poles at $s=-m-2l-1$ give rise to a series of the form $\sum_{l\ge0}q_lx^{m+2l+1}.$ (Here $p_l, $ $r_l,$ and $q_l$ are computable constants.) Combining these observations we find that 
$$(2\pi i)^m\sum_{n\ge1}\frac{\mu(n)\log(n)\bar{B}_m(nx)}{n^m}=C_m$$
$$+m!\sum_{\rho}K(\rho-m)e^{-(\rho-m)(\log(2\pi)+\log(x))}\Gamma(\rho-m)$$
$$+\Upsilon_{m-1}(x)+\sum_{l\ge0}((p_l+r_l\log(x))x^{m+2l}+q_lx^{m+2l+1}).$$
If we let $x$ become increasingly small we find the desired result upon inspecting the term $e^{-(\rho-m)\log(x)}$ in the sum over $\rho$ and then replacing $m$ with $k.$ That is, the equivalence of the Riemann hypothesis follows from the condition that the non-trivial zeros must have $\Re(\rho)=\frac{1}{2},$ and hence we estimate the sum by $O(e^{-(\frac{1}{2}-m)\log(x)})$ and negate terms from the last series we computed.\end{proof}
In Titchmarsh [10, pg.198, eq.(8.9.10)], we find the arithmetic function $b_r(n),$ $r\in\mathbb{N},$ and its Dirichlet generating function
\begin{equation} \frac{1}{\zeta^{r}(s)}=\sum_{n\ge1}\frac{b_r(n)}{n^s},\end{equation}
for $\Re(s)>1.$ We offer an analogue of Theorem 1.1 for $b_r(n).$
\begin{theorem} Let $k\ge1$ be a natural number, $r>1,$ and let $\bar{\Upsilon}_k(x)$ be a polynomial of degree $k.$ Put $C_k=\frac{k!}{(2\pi i)^k}(1+(-1)^k)\frac{1}{\zeta^{r-1}(k)},$ when $k>1,$ and $C_1=0.$ We have that the Riemann Hypothesis is equivalent to 
\begin{equation}\sum_{n\ge1}\frac{b_r(n)}{n^k}\bar{B}_{k}(nx)=C_k+\bar{\Upsilon}_{k-1}(x)+O(x^{k-\frac{1}{2}}),\end{equation}
as $x\rightarrow0^{+}.$
\end{theorem}
\begin{proof}The proof is identical to Theorem 1.1, but we estimate our integral with $|\zeta(s)|\le \zeta(\sigma)$ for $\sigma>1,$ which implies (since $\Re(s)>1$ is a zero free region for $\zeta(s)$)
$$\frac{1}{\zeta^{r-1}(\sigma)}\le \frac{1}{|\zeta^{r-1}(s)|}.$$
The computation involving the sum over $\rho$ includes computing the residues
$$R_{\rho,m,r}(x):=\lim_{s\rightarrow\rho-m}\frac{1}{(r-2)!}\frac{d^{r-2}}{ds^{r-2}}\left((s-\rho+m)^{r-1}K(s)(2\pi x)^{-s}\frac{1}{\zeta^{r-1}(s+m)}\Gamma(s)\right),$$
and further, if $\Re(\rho)=\frac{1}{2},$
$$\sum_{\rho}R_{\rho,k,r}(x)=O(x^{k-\frac{1}{2}}),$$
since the terms involving $\log(x)^{r-2}x^{k-\frac{1}{2}}$ decay faster as $x\rightarrow0^{+}.$
The polynomial $\bar{\Upsilon}_k(x)$ is defined as in the theorem, since this time the pole at $s=-m+1$ is simple. The remaining details are left to the interested reader. \end{proof}

\section{Some Further observations}
We mention some corollaries that are related to the $k=2$ series from (1.5). It can be obtained from (1.3) and the property $\{-x\}=1-\{x\}$ that 
\begin{equation} \bar{B}_1(x)^2-\frac{1}{12}=\frac{1}{2\pi^2}\sum_{n\ge1}\frac{e^{2\pi ixn}}{n^2},\end{equation}
\begin{equation} \sum_{n\ge1}\left(\frac{\bar{B}_1(nx)}{n}\right)^2-\frac{\pi^2}{72}=\frac{1}{2\pi^2}\sum_{n\ge1}d(n)\frac{e^{2\pi ixn}}{n^2},\end{equation}
where $d(n)$ is the number of divisors of $n.$ Furthermore, we have
\begin{equation}\sum_{n\ge1}\frac{\mu(n)}{n^2}\bar{B}_1(nx)^2=\frac{1}{2\pi^2}\left(e^{2\pi i x}+1\right),\end{equation}
\begin{equation}\sum_{n\ge1}\left(\frac{\mu(n)}{n}\bar{B}_1(nx)\right)^2=\frac{1}{12}\frac{\zeta^2(2)}{\zeta(4)}+\frac{1}{2\pi^2}\sum_{n\ge1}\frac{2^{v(n)}}{n^2}e^{2\pi i xn},\end{equation}
\begin{equation}\sum_{n\ge1}\frac{\lambda(n)}{n^2}\bar{B}_1(nx)^2=\frac{1}{12}\frac{\zeta(4)}{\zeta(2)}+\frac{1}{2\pi^2}\sum_{n\ge1}\frac{1}{n^4}e^{2\pi i xn^2},\end{equation}
where $v(n)$ is the number of different prime factors of $n$ [10, pg.5, eq.(1.2.8)], and $\lambda(n)$ is equal to $(-1)^k$ when $n$ has $k$ prime factors, counted according to the degree of the factor [10, pg.6, eq.(1.2.11)]. The series considered in [2,9] may be obtained from (2.1) (or (2.3)) by differentiating and then taking the real part. The series on the right hand side of (2.5) is related to a function of Riemann in considering a function which is non-differentiable [7]. Equivalent results could be imposed on (2.5) regarding twice differentiability. Recall that a periodic function with a continuous first order derivative has a uniformly convergent Fourier series. In [4, pg.216], several relevant results are given concerning convergence, one of which is that the series
$$\sum_{n\ge1}\frac{1}{n}e^{2\pi in^2 x},$$ is not convergent for all irrational $x.$ 
We close with a short proof concerning convergence of some of our arithmetic sums.
\begin{corollary} The function defined by
$$f(z)=\sum_{n\ge1}\left(\frac{\mu(n)}{n}\bar{B}_1(nz)\right)^2,$$ is $1$-periodic, converges locally uniformly on $\mathbb{H}:=\{z\in\mathbb{C}: \Re(z)>0\},$ and consequently is analytic there. Furthermore, the same can be said about (2.3).
\end{corollary}
\begin{proof} Since the Dirichlet series $\sum_{n\ge1}a_n/n^s,$ with $a_n=2^{v(n)}/n^2$ converges absolutely for $\Re(s)>-1,$ we have $a_n=O(n^{-1}).$ The result now clearly follows when comparing with the right side of (2.4). The series (2.3) follows from the comparison test.
\end{proof}

1390 Bumps River Rd. \\*
Centerville, MA
02632 \\*
USA \\*
E-mail: alexpatk@hotmail.com, alexepatkowski@gmail.com

\end{document}